%% file: main.tex
\documentclass[preprint,12pt]{elsarticle}
\usepackage{amssymb}
\usepackage{amsmath}
\usepackage[nolist,withpage]{acronym} 
\usepackage[capitalise]{cleveref}
\usepackage{multirow}
\usepackage[dvipsnames,table]{xcolor}
\usepackage{url}
\usepackage{mathtools}
\usepackage{amsthm}
\usepackage{romanbar}
\usepackage{datetime}
\usepackage{environ}
\usepackage{rotating}
\usepackage{booktabs}

\def\paramVessels{\ensuremath{n}}
\newcommand{\paramVesselSpeed}[1]{\ensuremath{v_{#1}}}
\newcommand{\paramVesselDraught}[1]{\ensuremath{\delta_{#1}}}
\newcommand{\paramVesselAmount}[1]{\ensuremath{a_{#1}}}
\def\paramStations{\ensuremath{m}}
\def\paramVesselStation{\ensuremath{C}}
\def\paramTypes{\ensuremath{f}}
\newcommand{\paramIncidentSevereity}[1]{\ensuremath{w_{#1}}}
\def\paramVesselIncident{\ensuremath{B}}
\def\paramZones{\ensuremath{z}}
\newcommand{\paramDistanceHarbourZone}[2]{\ensuremath{d_{#1 #2}}}
\def\paramVesselHarbourZone{\ensuremath{S}}

\newcommand{\paramIncidentZoneProbability}[2]{\ensuremath{q_{#1 #2}}}

\newcommand{\paramCorrectedWaterlevel}[1]{\ensuremath{\tilde{p}_{#1}}}
\def\paramTimes{\ensuremath{\mathcal{U}}}
\newcommand{\e}{\ensuremath{e}}

\def\N{\ensuremath{\mathbb{N}}}
\def\O{\ensuremath{\mathcal{O}}}
\def\NP{\ensuremath{\mathcal{NP}}}
\def\NPO{\ensuremath{\mathcal{NPO}}}
\def\R{\ensuremath{\mathbb{R}}}
\def\t{\ensuremath{t}}

\newtheorem{theorem}{Theorem}
\newtheorem{lemma}[theorem]{Lemma}
\newtheorem{corollary}[theorem]{Corollary}
\newtheorem{definition}[theorem]{Definition}

\journal{Operations Research for Health Care}

\begin{document}

\begin{frontmatter}

%% Title, authors and addresses

%% use the tnoteref command within \title for footnotes;
%% use the tnotetext command for theassociated footnote;
%% use the fnref command within \author or \address for footnotes;
%% use the fntext command for theassociated footnote;
%% use the corref command within \author for corresponding author footnotes;
%% use the cortext command for theassociated footnote;
%% use the ead command for the email address,
%% and the form \ead[url] for the home page:
%% \title{Title\tnoteref{label1}}
%% \tnotetext[label1]{}
%% \author{Name\corref{cor1}\fnref{label2}}
%% \ead{email address}
%% \ead[url]{home page}
%% \fntext[label2]{}
%% \cortext[cor1]{}
%% \affiliation{organization={},
%%             addressline={},
%%             city={},
%%             postcode={},
%%             state={},
%%             country={}}
%% \fntext[label3]{}

\title{Rescue Craft Allocation in Tidal Waters \\ of the North and Baltic Sea}

\author[inst1]{Tom Mucke}
\author[inst1]{Alexander Renneke}
\author[inst1]{Finn Seesemann}
\author[inst1]{Felix Engelhardt}

\affiliation[inst1]{organization={RWTH Aachen, Faculty of Mathematics, Computer Science and Natural Sciences, Teaching and Research Area Combinatorial Optimization},%Department and Organization
            addressline={Templergraben 55}, 
            city={Aachen},
            postcode={52062}, 
            state={NRW},
            country={Germany}}

\begin{abstract}
This paper aims to improve the average response time for naval accidents in the North and Baltic Sea. To do this we optimize the strategic distribution of the vessel fleet used by the \ac{dgzrs} across several home stations. Based on these locations, in case of an incoming distress call the vessel with the lowest response time is dispatched. A particularity of the region considered is the fact that due to low tide, at predictable times some vessels and stations are not operational.
In our work, we build a corresponding mathematical model for the allocation of rescue crafts to multiple stations. Thereafter, we show that the problem is \NP-hard. Next, we provide an Integer Programming (IP) formulation. Finally, we propose several methods of simplifying the model and do a case study to compare their effectiveness. For this, we generate test instances based on real-world data.
\end{abstract}

\begin{keyword}
Search and Rescue \sep Tides \sep Rescue Craft Allocation \sep Maritime  \sep Integer Programming \sep Facility Location
\end{keyword}

\end{frontmatter}

%% \linenumbers

\section{Introduction}\label{sec:introduction}
% General background, motivation
The water territory of Germany is home to a multitude of maritime traffic. These ships transporting both people and goods are exposed to a variety of dangers, ranging from human error to extreme weather, all of which may lead to the necessity of calling for external aid. For the water territory of Germany, the \ac{dgzrs} is primarily responsible for delivering aid  \cite{BundesministeriumderJustiz.19941111}. In 2022 alone, over 2000 instances of vessels in need of assistance were recorded  \cite{DGzRS.2023}.

% Objective of this work
Oftentimes, the speed at which an adequate rescue vessel arrives can make a difference between life and death, see \cite{timematters}. 
Thus, ensuring a timely arrival of rescue crafts is an important factor to maritime safety.
The following work uses mathematical optimization of vessel locations to reduce average arrival times, and thereby contribute to saving lives.
This is done by focusing on the strategic decision made by the \ac{dgzrs} of where to place the vessels of their fleet to ensure the fastest possible average response time for future incidents. 

% Description of the process in question
The rescue process, as far as it is relevant to our work, is the following: A ship somewhere near the German coastline suffers an accident that necessitates help from outside forces. It then calls the control center of the \ac{dgzrs} and details both its position and type of incident. The control center in return checks the list of their available response vessels and decides which of them to send. This vessel is then deployed to help the ship with the incident.

% Problem definition
Thus, the \ac{rcap} consists of allocating a set of different vessel types to stations in order to ensure minimal response times to maritime incidents, given a region consisting of zones in which incidents occur as well as stations at fixed positions.
This is hindered by the fact that each station can only house specific types of crafts as well as that at (predictable) times some stations are inoperable due to the tides. Furthermore, some regions are more incident-prone and require more attention than others.

% Role of Tides und single crafts
Additionally, since the \ac{dgzrs} boats are manned by volunteers, each harbor has to house a single rescue craft.
For example, the harbour of Juist dries up twice a day. During that time, the region their station normally covers needs to be covered by the neighbouring stations. This makes it less desirable to position the fastest vessel available there. 
However we can not permanently place the station from Juist elsewhere since the crew of the ship consists of volunteers living on the island, who can not be resettled.

% Our contribution
This works contributions are twofold. First, it is novel insofar as, to the best of our knowledge, there have been not previous attempts to use mathematical optimization for search-and-rescue craft allocation in either the North or the Baltic Sea. 
Second, the consideration of tides is a novel attribute that specifically matters in seas with a large tidal range, e.g., the North and Baltic Seas.
Additionally, we provide a working implementation of our solution algorithms and the corresponding data.

% Structure of the paper
The remainder of the paper is structured as follows: in \cref{sec:literature} we introduce the \ac{rcap} and give a brief overview of related works. Past this, in \cref{sec:definition} we formally introduce the \ac{rcap}. Afterward, in  \cref{sec:complexity}, we prove that \ac{rcap} is \NP-hard and examine why. In \cref{sec:exact_sol}, we formulate the \ac{rcap} as an \ac{mip} and we consider simplifications with the aim of reducing computation times with minimal precision losses. Having done this, in \cref{sec:computationalstudy} we conduct a case study on the effect of these simplifications across several test instances. We discuss the results in \cref{sec:discussion}, and summarize our findings and give avenues for further research in \cref{sec:conclusion}.

\section{Related Work}\label{sec:literature}
In the following, we discuss research that includes optimization of search and rescue vessel locations.

% Single and multiple levels
The motivation driving our work is also the basis of \cite{Azofra.2007} in which the aim is to find suitable criteria to find the optimal placement of a single rescue vessel. To the best of our knowledge, \cite{Azofra.2007} is also the first paper to address the problem rescue vessel placement. The problem of optimal assignment of a single craft can also be found in \cite{Afshartous.2009} which places greater focus on the size and form of the zones generated as well as transforming the historical numbers of incidents into probabilistic values.
In comparison, \cite{jin2021} do not limit themselves to placement of a single vessel. However, they make the same assumption as us that only a single vessel is assigned to each harbour. They then solve the problem through a multistage approach based on $k$-means and nature-inspired heuristics. At the same time, their problem strongly differs from the one in this work, as it focuses on dynamic duty points at sea.

% Integer Programming
Another similar model and solution approach, including the construction and solving of an \ac{mip}, is given by \cite{Razi.2016} and \cite{Wagner.2012}. While the former analyzes historical data given by the Turkish Coast Guard, the later concentrates on the development of a practical tool for the US Coast Guard. Both differ from our work in that their aim is to minimize the deviation from given values of budget and operating hours instead of minimizing individual response times.
\cite{Razi.2016} is further expanded by \cite{Karatas.2021} in which one of the authors explores several ways of improving the \ac{mip} in terms of realism. Besides introducing a more diverse arsenal of rescue crafts it also considers how to implement and react to uncertainty in terms of incidents.
Another \ac{mip}-based approach is provided in \cite{chen2021}. Here, the authors use a two-stage approach to solve the \ac{mip}, since the original formulation does not perform well computationally.
For small instances, enumeration may also be possible, as showcased in \cite{Jung2019}.

% Recent stuff
Recent research by \cite{Hornberger.2022} focuses on the pacific ocean area the US Coast Guard is responsible for. While it shares many similarities to our work, it differs in the modeling of incidents. While we consider the incidents as being few enough that it can be assumed a needed vessel is at its harbour when the distress call comes in, \cite{Hornberger.2022} assigns every incident a certain amount of hours needed during which the responding vessel can not move to the next incident. Their work also includes the possibility of relocating vessels for a certain price.
Finally, there is some overlap with 
\cite{ma2024}, who only assign a number of homogeneous boats to sections, but do so addressing uncertainties through robust optimization.

% Adjacent fields
The field of research regarding maritime search and rescue also has many works with a similar motivation but a different approach in terms of model and aim. Other related work primarily focuses on on the possible causes of incidents, and on how to compare and combine their severity \cite{Zhou.2022}. \cite{FeldensFerrari.2020} also deals with maritime search and rescue but in terms of searching a given area of a single incident.

An overview of the different sources and their properties is given in table \cref{tab:source:comp}.

\input{table_citations.tex}

% Summary
As the first column shows, almost all sources have given stations and the \textit{zones} column signals that nearly all of them are working with zones to describe and solve the optimization problem. The \textit{vessel types} column indicates if a source considered vessels with different characteristics like we do or if the focus of allocation is on another aspect. The last two columns reveal that solving the problem is mostly done through a \ac{mip}-solver rather than by using a specific algorithm.

For a broader overview over the topic of \ac{SAR} operations, we refer to the state of the art paper by \cite{Akbayrak2017}. The authors note that in general, the subject of this work, the allocation of assets (vessels) to stations (harbors) is one of multiple closely related fields, i.e., location modelling of \ac{SAR} stations, allocation modelling of \ac{SAR} assets, risk assessment modelling of \ac{SAR} areas, and search theory and \ac{SAR} planning modelling. For, a general discussion of \ac{SAR} and its connection to other medical facility locations problems, including stochastic variations, we refer to \cite{Pelot2015}.

\section{Model}\label{sec:definition}
Next we formulate our model for the \ac{rcap}. Here, consideration of the tides is the most notable difference from the models of the previously listed sources.

\subsection{Parameters}
First we define the parameters needed for a single instance. For each instance, this includes information on the available vessels, stations and zones.

Regarding the available vessels, a defining trait is the total number of different types of vessels $\paramVessels \in \mathbb{N}$, the speed of each type of vessel $\paramVesselSpeed{1},\dots, \paramVesselSpeed{\paramVessels} \in \mathbb{N}$, the draught of each vessel $\paramVesselDraught{1}, \dots, \paramVesselDraught{\paramVessels} \in \R^+$ as well as the amount of available vessels for each type $\paramVesselAmount{1}, \dots, \paramVesselAmount{\paramVessels} \in \mathbb{N}$.
In terms of stations, we need to know their total number $\paramStations \in \mathbb{N}$ and, to represent the relationship between vessel types and stations, a set $\paramVesselStation \subseteq \underline{\paramVessels} \times\underline{\paramStations}$,\footnote{$\underline{n} \coloneqq \{1,\;\dots,\;n\}$ for $n \in \mathbb{N}$} where $(i, j) \in \paramVesselStation$ means that a vessel of type $i$ can be positioned at station $j$. 

Due to the tides the water level at each of the stations changes regularly which leads to vessels with a too high draught being unable to leave. To model this, we consider data of the tides in a period $\mathcal{T}$. We then consider each point of time $\t \in \mathcal{T}$ and collect all usable combinations of vessels and stations $\e \subseteq \underline{\paramStations} \times \underline{\paramVessels}$ at that point of time $\t$. We define a uncertainty set $\paramTimes \subseteq \mathcal{P}(\underline{\paramStations} \times \underline{\paramVessels})$ that contains all usable configurations $\e$. 

The incidents the vessels respond to are grouped into $\paramTypes \in \mathbb{N}$ types with severities $\paramIncidentSevereity{1}, \dots, \paramIncidentSevereity{\paramTypes} \in \mathbb{R}$. Since some incident types may require specific features of a vessel, such as enough weight and power to tow a heavy vessel, the set $\paramVesselIncident \subseteq \underline{\paramVessels} \times \underline{\paramTypes}$ represents compatibility between vessel and incident types.

The territory to be covered is divided into $\paramZones \in \mathbb{N}$ zones, which have certain distances to the stations represented by \paramDistanceHarbourZone{j}{r} for the distance between station $j$ and zone $r$. Because not every vessel can reach every zone from every station (e.g.,  due to tank size or offshore unsuitability), we have a set $\paramVesselHarbourZone \subseteq \underline{\paramVessels} \times \underline{\paramStations} \times \underline{\paramZones} $ representing the compatible vessel-station-zone-combinations. 

\subsection{Feasible Solutions}
A feasible solution consists of two parts. The first one represents the allocation of the vessel types to the stations. Let $G$ be a graph whose nodes represent all vessel types and all stations and whose edges are given by the set of compatible combinations $\paramVesselStation$. A solution consists of a (not necessarily optimal) bipartite $b$-Matching $M$ with $b(v) = \paramVesselAmount{i}$ if $v \in V(G)$ represents vessel type $i$ and $b(v) = 1$ if it represents a station. A vessel type $i$ is assigned to a station $j$ if the edge between their nodes is part of $M$.

The second part of a feasible solution ensures that every incident can be responded to. It is a mapping $u: \; \underline{\paramTypes} \times \underline{\paramZones} \times \paramTimes \rightarrow \underline{\paramStations}$ that assigns every combination of incident type $k$, zone $r$ and water state $e$ a responding station $j$ and, combined with $M$, a responding vessel type $i$. This mapping has to adhere to the limitations outlined above: The vessel type must be able to help with the incident type meaning $(i,k) \in \paramVesselIncident$, it must be able to traverse the distance between harbour and zone meaning $(i,j,r) \in \paramVesselHarbourZone$ and the station must have a high enough water level to be operational for the vessel so $(j, i) \in \e$.

\subsection{Objective Function}
The aim of our model is to minimize the average weighted response time to incidents at sea. For every combination $(k,r,e)$ of incident type, zone and water state and their responding station $j=u(k,r,e)$, there is a distance \paramDistanceHarbourZone{j}{r}.
Let $i$ be the vessel type of the node matched to the node of $j$ in $M$. Dividing the distance by the speed $\paramVesselSpeed{i}$, a response time can be calculated. The average weighted response time of a solution is defined as the sum of all response times, each multiplied by the severity $\paramIncidentSevereity{k}$ of the incident type.
\begin{equation}\label{eq:opt_obj_func_stochastic}
\min_{\substack{(k,r,e) \in \paramTypes \times \paramZones \times \paramTimes\\ j=u(k,r,e)}} 
\mathbb{E} \left(
\frac{\paramDistanceHarbourZone{j}{r}}{\paramVesselSpeed{i}} {\paramIncidentSevereity{k}} 
\right).
\end{equation}
We assume that for every zone $r$ and incident type $k$ there is a frequency $\paramIncidentZoneProbability{k}{r} \in [0,1]$ of an incident type occurring in the given zone. To track how often each usable configuration $\e$ appears in the whole time period $\mathcal{T}$, we set $\paramCorrectedWaterlevel{e} \in [0,1]$ to the percentage of time it occurs. Note that the equation  
$$\sum_{e \in \paramTimes} \paramCorrectedWaterlevel{e} = 1$$ 
is true, since at every point of time $x \in \mathcal{T}$ exactly one configuration $e$ occurs.
Thus, we can reformulate the objective as 
\begin{equation}\label{eq:opt_obj_func}
    \min \sum_{\substack{(k,r,e) \in \paramTypes \times \paramZones \times \paramTimes\\ \\ j=u(k,r,e)}} \frac{\paramDistanceHarbourZone{j}{r}}{\paramVesselSpeed{i}} \paramIncidentSevereity{k} \paramIncidentZoneProbability{k}{r} \paramCorrectedWaterlevel{e}.
\end{equation}
To summarize, an overview of all parameters and variables is given in \cref{tab:variables}.

\input{vartable}

\section{Complexity} \label{sec:complexity}
In this section, we examine the complexity of the \ac{rcap}. We show that the feasibility problem corresponding to \ac{rcap} is  \NP-hard by reduction from \ac{x3c}. Based on this, we argue that the \ac{rcap} is $\mathcal{NPO}$-complete. 

\begin{definition}\label{def:frcap}
    Let $A$ be an instance of the \ac{rcap}. We define the \ac{frcap} as the problem of finding a feasible solution for $A$, or proving that no such solution exists.
\end{definition}
Note that the \ac{frcap} is a decision problem, whereas \ac{rcap} is an optimization problem. We do all proofs for the decision problem and then extend them to the optimization problem.

\begin{lemma}\label{lemma:np}
    The \ac{frcap} is in \NP.
\end{lemma}
\begin{proof}
    A solution must consist of a function $M: \; \underline{\paramStations} \; \rightarrow \; \underline{\paramVessels}$ assigning every station in $\underline{\paramStations}$ a vessel in $\underline{\paramVessels}$ as well as a function $u: \; \underline{\paramTypes} \times \underline{\paramZones} \times \paramTimes \rightarrow \underline{\paramStations}$ assigning every water state $e$, zone in $\underline{\paramZones}$ and incident type in $\underline{\paramTypes}$ a responding station in $\underline{\paramStations}$.

Given these, a in a feasible solution harbours only house compatible vessels, meaning $(j,M(j)) \in \paramVesselStation$ for all $j \in \underline{\paramStations}$, which can be done in $\O(\paramStations)$.
Furthermore we need to:
\begin{itemize}
    \item Verify that the vessel of the responding station is operable during the given water state, i.e., the harbour vessel combination must be in $e$.
    \item Ensure, that the vessel is equipped to deal with the given incident, meaning $(M(u(k,r,e)), k) \in \paramVesselIncident$ must be true for all $k \in \underline{\paramTypes}, \; r \in \underline{\paramZones}, \; e \in \paramTimes$.
    \item Check that the given zone in $\paramZones$ is reachable from the specified station by the responding vessel, meaning $(M(u(k,r,e)), u(k,r,e), r) \in \paramVesselHarbourZone$ must be true for all $k \in \underline{\paramTypes}, \; r \in \underline{\paramZones}, \; e \in \paramTimes$.
\end{itemize}
Since each of these categories requires either $\O ( \paramStations )$ or $\O ( \paramTypes \paramZones \left| \paramTimes \right|)$ checks with a runtime of $\O(1)$ each, in total we have a polynomial runtime of $\O (\paramStations + \paramTypes \paramZones \left| \paramTimes \right|)$.
\end{proof}

\begin{lemma}\label{lemma:nphard}
    \ac{frcap} is \NP-hard, even for only two vessel types, disregarded differing vessel-ranges, vessel-station incompatibilities, incident types and water states.
\end{lemma}
\begin{proof}
We use the \ac{x3c} to prove the \NP-hardness of the \ac{frcap}. Each instance of the \ac{x3c} 
consists of a set $X$ with $q \coloneq \frac{|X|}{3} \in \N$ and a set $D \subseteq \{x \subseteq 
X : |x|=3\}$. For this instance the decision problem is whether there exists a subset $A 
\subseteq D$ such that $|A|=q$ and $\bigcup A = 
X$. This problem is a known \NP-complete problem, see \cite[221]{Garey.1979}.

Given an instance of the \ac{x3c}, we now construct an equivalent \ac{rcap} instance. To do so we create one zone per element of $X$, one station for each set in $D$, and two vessel types \Romanbar{1} and \Romanbar{2}. We only create a singular incident type of severity $1$ and probability $1$ in every zone. The distance from any station to any zone is set to $1$ and the speed of all vessels is set to $1$ as well. There are $q$ vessels of type \Romanbar{1} and $|D|-q$ of type \Romanbar{2} available.
Each vessel is operable in every station at all times. All vessels can be assigned to any of the stations. Vessels of type \Romanbar{2} can not reach any zone from any of the harbours, vessels of type \Romanbar{1} positioned at the station corresponding to $d \in D$ can reach all zones corresponding to one of the elements in $d$. Every vessel type is capable of assisting with the single incident type.

To show such an instance is equivalent to the original we will prove that a feasible solution to the \ac{rcap}-instance exists if and only if the \ac{x3c} instance is a yes-instance. 

First, starting with a solution to the \ac{rcap}, let $A$ be the set of all elements of $D$ which correspond to stations that have a vessel of type \Romanbar{1} assigned to them. This is a solution to the \ac{x3c} as there are $q$ ships of type \Romanbar{1}, so $|A| = q$ and all $3q$ zones are covered, while each station can cover up to $3$ zones with a vessel of type \Romanbar{1} and $0$ zones with type \Romanbar{2}. Due to $\paramZones = 3q$ this means that each station with a vessel of type \Romanbar{1} assigned covers exactly $3$ zones, thereby all of the elements of $X$ occur in exactly one set in $A$ and $\bigcup A = X$.

Second, starting with a solution $A$ to the instance of the \ac{x3c}, we assign vessels of type \Romanbar{1} to all stations corresponding to an Element in $A$ and vessels of type \Romanbar{2} to the remaining stations. Then for each zone with corresponding element $x \in X$ exactly one station with a vessel of type \Romanbar{1} is capable of responding to incidents of the singular type in that zone as $x \in \bigcup A$. This gives a feasible (and optimal) solution.

As the two instances are equivalent and the transformation is polynomial this proves that the \ac{frcap} is \NP-hard.
\end{proof}

\begin{theorem}
    \ac{frcap} is \NP-complete, even for only two vessel types, disregarded differing vessel-speeds, vessel-station incompatibilities, incident types and water states.
\end{theorem}

\begin{proof}
    This follow immediately from \cref{lemma:np} and \cref{lemma:nphard}.
\end{proof}

\begin{corollary}
    \ac{rcap} is \NPO-complete, even for only two vessel types, disregarded differing vessel-speeds, vessel-station incompatibilities, incident types and water states.
\end{corollary}

\begin{proof}
    This follow immediately from the \NP-completeness of \ac{frcap}.
\end{proof}

Note that we could equally formulate \cref{lemma:nphard} in terms of ranges, not inabilities to reach certain zones.
Furthermore, note that the two different vessel types and their (in)ability to reach the incident zones were key to our reduction.
This difference can alternatively be replaced by a difference in speed:

\begin{corollary}
        \ac{rcap} is \NPO-complete, even for only two vessel types, disregarded differing vessel-ranges, vessel-station incompatibilities, incident types and water states.
\end{corollary}

\begin{proof}
We assign vessel type $\Romanbar{1}$ a speed of $v_{\Romanbar{1}}=1$ and vessel type $\Romanbar{2}$ a speed of $v_{\Romanbar{2}}=0.5$ (and set all distances to $\paramDistanceHarbourZone{j}{r}=1$). The goal equivalent to finding a solution for the \ac{x3c}-instance is to find a solution for the constructed \ac{rcap}-instance with a total response time of at most $3q$. Since every one of the $3q$ zones contributes either $1$ or $2$ to the total response time depending on the vessel responding, a total of $3q$ is equivalent to using only the $q$ vessels of type \Romanbar{1} to respond to the incidents.
\end{proof}

In conclusion, there is no singular aspect of the \ac{rcap} responsible for its complexity because any feature used in the transformation by itself can be replaced by a combination of the other parameters.

\section{Integer Program for \ac{rcap}}
\label{sec:exact_sol}
The problem of rescue craft-allocation can be modeled as a (binary) \ac{mip} as seen below. 
In this context, the variable $x_{ij} \in \{0,1\}$ represents the decision to assign a ship of type $i \in \underline{\paramVessels}$ to station $j \in \underline{\paramStations}$ while the variable $y_{ijkre}$ represents the decision to let a ship of type $i$ positioned at station $j$ attend to incidents of type $k \in \underline{\paramTypes}$ occurring in zone $r \in \underline{\paramZones}$ during the water state $e \in \paramTimes$. 
In order to avoid creating unnecessary variables due to the constraints given by $\paramVesselStation,\paramVesselIncident, \paramTimes$ and $\paramVesselHarbourZone$ we define the sets
\begin{align}
    \mathcal{X} \coloneq \{&(i,j) \in \underline{\paramVessels} \times \underline{\paramStations} : (i,j) \in C\} \text{ and }\\
    \mathcal{Y} \coloneq \{&(i,j,k,r,e) \in \underline{\paramVessels} \times \underline{\paramStations} \times \underline{\paramTypes} \times \underline{\paramZones} \times \paramTimes  : \label{setY} \\
    \nonumber&(i,j) \in \mathcal{X}, (i,k) \in \paramVesselIncident, (j, i) \in e, (i,j,r) \in \paramVesselHarbourZone\} \\
    \bar{\mathcal{Y}}_{ij} \coloneq \{&(k,r,e) : (i,j,k,r,e) \in \mathcal{Y}\}.
\end{align}

Using these we can describe our problem as a \ac{mip} as follows
\begin{subequations}
\begin{align}
    \nonumber &\min && \rlap{$\displaystyle\smashoperator[l]{\sum_{(i,j,k,r,e) \in \mathcal{Y}}} 
    \frac{d_{jr}}{\paramVesselSpeed{i}}
    q_{kr}
    \paramCorrectedWaterlevel{e}
    \paramIncidentSevereity{k}
    y_{ijkre}$} && \\
    &\text{s.t.} && \smashoperator[l]{\sum_{\substack{i,j:\\(i,j,k,r,e) \in \mathcal{Y}}}} y_{ijkre} \geq 1 && \forall k \in \underline{\paramTypes},r \in \underline{\paramZones},e \in \paramTimes \label{mip1} \\
    %& && y_{ijkre} \leq x_{ij} && \forall (i,j,k,r,e) \in \mathcal{Y}  \\
    & && \smashoperator[l]{\sum_{(k,r,e) \in \bar{\mathcal{Y}}_{ij}}} y_{ijkre} \leq |\bar{\mathcal{Y}}_{ij}| x_{ij} && \forall (i,j) \in \mathcal{X}\label{mip2}\\
    & && \smashoperator[l]{\sum_{\substack{j:\\(i,j) \in \mathcal{X}}}} x_{ij} \; \leq \; a_i && \forall i\in \underline{\paramVessels} \label{mip4} \\
    & && \smashoperator[l]{\sum_{\substack{i:\\(i,j) \in \mathcal{X}}}} x_{ij} \leq 1 && \forall j \in \underline{\paramStations} \label{mip3} \\
    & && x_{ij} \in \{ 0,1 \} && \forall (i,j) \in \mathcal{X} \label{mip5} \\
    & && y_{ijkre} \in \{ 0, 1 \} && \forall (i,j,k,r,e) \in \mathcal{Y} \label{mip6}
\end{align}
\end{subequations}
Our objective function (which is similar to \cref{eq:opt_obj_func}) is the sum of the vessel-station-zone-state-incident-assignments, each weighted by the incident-frequency $q_{kr}$, the probability of the water state $\paramCorrectedWaterlevel{e}$, the incident-severity $\paramIncidentSevereity{k}$ and most importantly the traveling time $\frac{d_{jr}}{\paramVesselSpeed{i}}$.
Constraint \ref{mip1} is used to ensure every zone-time-incident-combination is attended to and Constraint \ref{mip2} ensures that a vessel is stationed at the station it is sent out from.
Constraint \ref{mip4} limits the amount of vessels assigned to the number of vessels available for every vessel type. Constraint \ref{mip3} ensures that every station has at most one vessel assigned to it. 
% This constraint can be easily changed to allow more than one vessel per station by exchanging the $1$ with a parameter for the capacity of each station or it can be adjusted to enforce exactly one vessel per station by exchanging the $\leq$ for $=$. 
% move to outlook, if at all

Initial computational testing showed that the explicit, stochastic version of the \ac{mip} provided above performs badly in practice and starkly limits the possible resolution of zones. Thus, we also provide two subject-specific simplifications, which we  evaluate in \cref{sec:computationalstudy}.

\subsection{Corrected water levels for Stations and Vessels} \label{subsec:algo:heuristic:stationwater}
As the Baltic and the Northern Sea are continuous bodies of water within a limited geographical area, for both seas the tides at different locations are strongly correlated. We validated this for the tide data used in this work, as shown in \ref{app:sec:tidecorrelations}. The data sourcing is covered in more detail in \cref{sec:computationalstudy}.
Based on this, we can make the simplifying assumption that during the time frame a specific station-vessel combination is operable, all combinations that are overall more frequently available are also operable. 

This means we sort the station-vessel combinations by relative availability and instead of considering all possible combinations of operable stations (the set $\paramTimes$) we only focus on situations where the most frequently available station-vessel combinations are operable.
In practical terms, we calculate
\begin{equation*}
    \paramCorrectedWaterlevel{(j,i)} \coloneq \sum_{e \in \paramTimes} \mathbf{1}_e((j,i)) \paramCorrectedWaterlevel{e}
\end{equation*}
as  the relative availability of the station-vessel combination $(j,i) \in \underline{\paramStations} \times \underline{\paramVessels}$ where $\mathbf{1}$ is the indicator function.
We define the set of availabilities as 
\begin{equation*}
    P=\{\paramCorrectedWaterlevel{(j,i)}:(j,i) \in \underline{\paramStations} \times \underline{\paramVessels}\} \cup \{0, 1\}
\end{equation*}
and use the intervals 
\begin{equation*}
    \bar{P}=\{[p_1, p_2] \in P^2: p_1 < p_2 \land \nexists p_3: (p_1 < p_3 < p_2)\}
\end{equation*}
to replace $E$. This means that in the interval $[p_1, p_2]$ we assume that all station-vessel combinations with a probability above $p_1$ are available.
Due to $\left|E\right| \in \mathcal{O}(2^{\paramStations \paramVessels})$ and $\left| \bar{P} \right| \in \mathcal{O}(|P|)=\mathcal{O}(\paramStations \paramVessels)$, this simplification can significantly decrease the input size of an instance. The inaccuracy induced by this simplification is based on the difference in water levels across stations at the same point of time which is relatively small.

To adjust the \ac{mip} from \cref{sec:exact_sol} it is mainly necessary to adjust \cref{setY} to
\begin{align} \label{simplify0}
    \mathcal{Y} \coloneq \{&(i,j,k,r,t) \in \underline{\paramVessels} \times \underline{\paramStations} \times \underline{\paramTypes} \times \underline{\paramZones} \times \bar{P}  :\\
    \nonumber&(i,j) \in \mathcal{X}, (i,k) \in \paramVesselIncident, \paramCorrectedWaterlevel{(j,i)} \geq \min (t), (i,j,r) \in \paramVesselHarbourZone\} 
\end{align}
and all occurrences of $\paramTimes$ accordingly.
Additionally the target function must be changed to 
\begin{equation*} \label{simplify1}
    \min \sum_{(i,j,k,r,t) \in \mathcal{Y}}
    \frac{d_{jr}}{\paramVesselSpeed{i}}
    q_{kr}
    (\max(t)-\min(t))
    \paramIncidentSevereity{k}
    y_{ijkrt},
\end{equation*}
which is still a linear function since $\max(t)$ and $\min(t)$ are constants.

\subsection{Corrected water levels for Stations} \label{subsec:algo:heuristic:station}
The simplification above can be further expanded on by averaging the water level of a station and deleting the dependence of the vessel stationed there. Given the $\paramCorrectedWaterlevel{(j,i)}$ of the previous section, for a fixed station $j$ we calculate 
\begin{equation*}
    \paramCorrectedWaterlevel{j} \coloneq \frac{
    \sum_{i \in \underline{\paramVessels}} \paramCorrectedWaterlevel{(j,i)} \paramVesselAmount{i}
    }{
    \sum_{i \in \underline{\paramVessels}} \paramVesselAmount{i}
    } 
\end{equation*}
as relative availability by scaling the relative availability in combination with every vessel type by the vessel type amount. These further simplified water levels $\paramCorrectedWaterlevel{j}$ can be used to replace the inequality  $\paramCorrectedWaterlevel{(j,i)} \geq \min(t)$ in \cref{simplify0} with $\paramCorrectedWaterlevel{j} \geq \min(t)$. This simplification further reduces the size of $P$ to $\mathcal{O}(m)$ instead of $\mathcal{O}(nm)$ by averaging the draught values of all vessel types. The \ac{mip} can be adjusted similarly to \cref{subsec:algo:heuristic:stationwater} while changing the corresponding indices from $(j,i)$ to $j$.

\section{Computational Study}\label{sec:computationalstudy}

We tested the \ac{mip} variations on several instances of the problem. In the following, we give an overview of instance generation, the data used, assumptions made about the data, and the computational setup. We then explain how solutions were validated.
The code and data are publicly available through GitHub \cite{Mucke_RCA_2024}.

\subsection{Available data}
An instance consists of information about the vessel types, the stations, the incidents, the tide levels, and the relations thereof.

\subsubsection{Data on vessel types, stations and incidence types}
The information about the vessel types and stations is based on \cite{DGzRS.2023b}, where the \ac{dgzrs} lists the vessels and stations currently used. At the time of writing there are $11$ types of vessels and $55$ rescue stations. Every vessel type listed has a specifications sheet detailing information such as the speed of the vessel in knots, its reach depending on its speed in nautical miles, its draught in metres and information about its equipment such as material for firefighting or towing ships of several sizes. 
%Furthermore, the amount of vessels of every type can be calculated by summing up the vessels that are currently employed at the stations. 
As the actual reach is dependent on the speed, which our model does not take into account, we always use the maximum given reach when deciding whether a vessel is capable of reaching a zone from a certain station.
Regarding the stations, we extracted their coordinates. The positions of the stations are shown in \cref{fig:stations:positions}. The vessel-station-compatibilities were randomly generated by giving each combination the probability $p=0.9$ of being allowed.
\begin{figure}[h!tb]
    \centering
    \includegraphics[width=\textwidth]{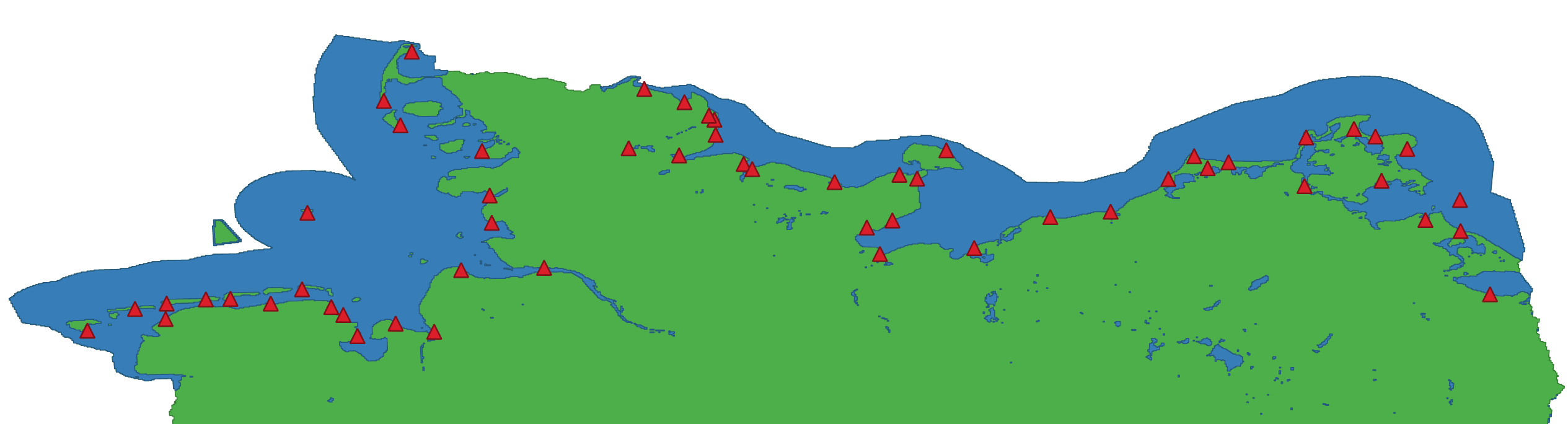}
    \caption{The red triangles show the positions of the $55$ \ac{dgzrs}-stations in the Northern and Baltic Sea.}
    \label{fig:stations:positions}
\end{figure}
While the real world data gives indicators such as vessel length and current crew regarding this compatibility, it is insufficient to create clear rules of which combinations are appropriate.

Using the equipment list of the vessel types we manually define five incident types, requiring firefighting equipment, pumping equipment, a secondary craft, a board hospital, or requiring first aid (which every vessel offers). Additionally we create several incident types requiring towing vessels of various sizes. Using the equipment lists mentioned above we decide which vessels are able to respond to them. At instance generation, the severity of every incident type is set to a random value in $[0,1]$. The frequency of an incident type occurring in a given zone is also randomly chosen in a two-step process: First if the incident type has a non-zero chance of occurring (with a value of $p=0.4$ for positive) and if it does, a random value in $[0,1]$ for the actual frequency. 
%While some vessel types are fit with additional equipment (like a helicopter working deck for the largest vessel type), this is not considered as part of our instance due to the fact that the uniqueness of the equipment would render far away incidents of this type impossible to tend to which in turn would lead to infeasible instances.

\subsubsection{Geographical data}
In our study, we focus on the German territorial waters as detailed in \cite{BundesministeriumderJustiz.19941111}. In our implementation we work with the data from \cite{BKG.2019} using QGIS (see \cite{QGISDevelopmentTeam.2023}) and its integrated Python-interface PyQGIS. We use their $100$ meter grid to filter for those map squares which are completely covered by water and then dissolving this grid by calculating the connected components.
The two biggest connected components represent the North and Baltic Sea quite closely because rivers and islands are excluded due to the $100$ meter grid. One main flaw however is that the \textit{Tiefwasserreede}, a German exclave as detailed in \cite{BundesministeriumderJustiz.19941111}, also is excluded.

\subsubsection{Generation of zones}
In order to generate zones we place $1000$ random points in the polygon using a PyQGIS script. For the distance calculation between these zones and stations we use the direct distances between two points on the surface of the earth. This means we disregard possible obstacles like islands and restricted areas and also disregard possible currents and differences between tidal levels. %For an actual deployment these should be taken into consideration, however this requires additional proper nautical navigation software with a usable API. 
To calculate the distances on the surface of the earth we used the haversine formula (implementation from the python haversine package\footnote{\url{https://github.com/mapado/haversine}}) with an earth radius of $6371$ kilometers.
These distances combined with the range of the vessel types (divided by two to account for the way back) decide if the vessel can travel between the zones and the stations.
In order to make these instances solvable, similarly to \cite{Hornberger.2022}, we then apply $k$-means clustering on the Geo locations. \Cref{fig:heuristics:kmeans:exp:1} shows an example of the generated zones ($500$ meter grid instead $100$ meters for better visualization).
\label{subsec:algo:heuristic:kmeans}

\begin{figure}[h!tb]
    \centering
    \includegraphics[width=\textwidth]{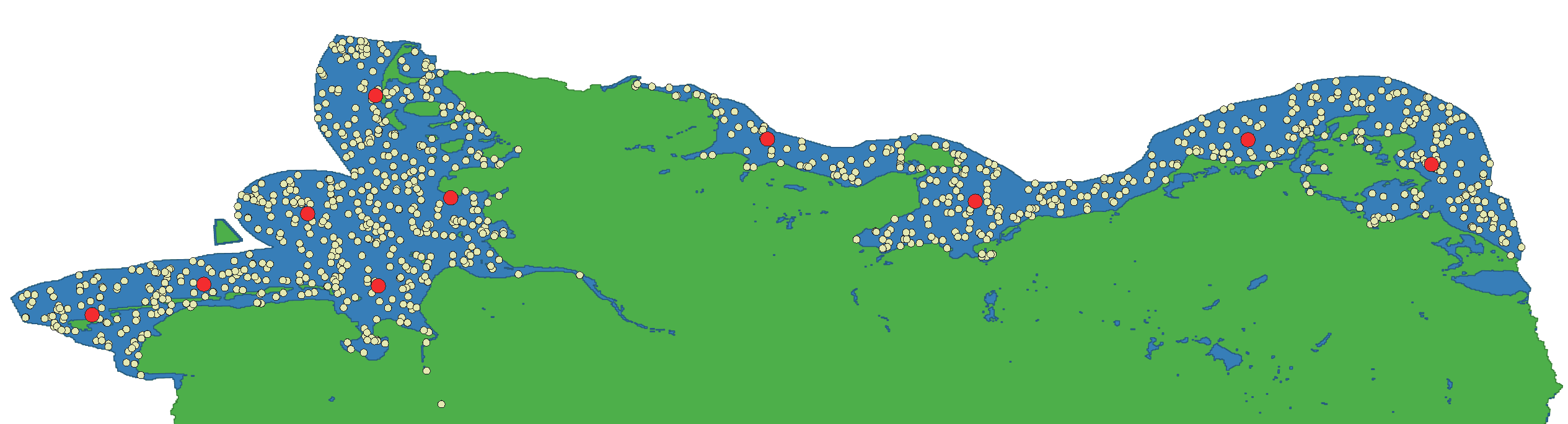}
    \caption{Incident generation across North and Baltic Sea. Yellow dots represent zones generated (amount: $1000$). Red dots represent clusters of zones (amount: $10$).}
    \label{fig:heuristics:kmeans:exp:1}
\end{figure} 

When location points are combined to one point $x$ we use the mean value of the incident rates of the old points for $x$.
%The first positive effect this has is a reduction of the incidents to be considered which in turn reduces the input size of the instance.
A side effect is smoothing the location values. Previously, a location point close to multiple incident locations would be considered completely harmless since no incidents happened exactly there. After clustering, these incidents indicate that this adjacent point should also be considered dangerous.
%The trade-off in using the $k$-means clustering is an inaccuracy when calculating distances between stations and incident locations.
For the implementation of $k$-means clustering, a deterministic KMeans-method provided by the sklearn-library was used, see \cite{Pedregosa.2011}. 

\subsubsection{Data on Tide levels}
For determining adequate tide levels, we consult \cite{GeneraldirektionWasserstraenundSchifffahrt.2023} where the measurements of water gauges across both North and Baltic Sea are published as well as their position.
These can be seen in \cref{fig:waterlevels}.

\begin{figure}[h!tb]
    \centering
    \includegraphics[width=\textwidth]{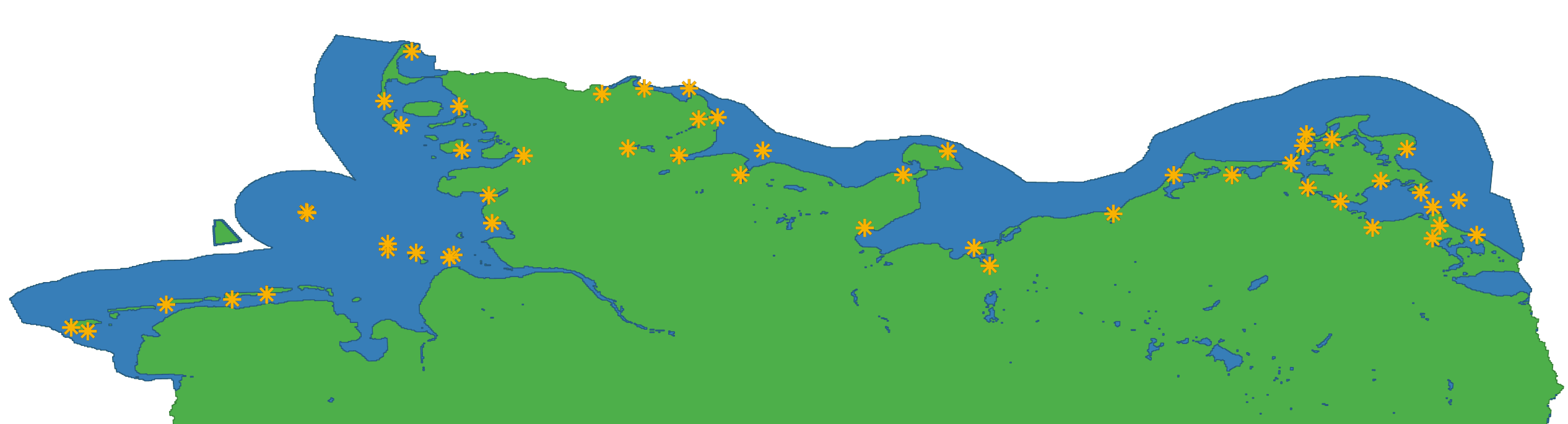}
    \caption{Locations of the gauges used to determine the water levels}
    \label{fig:waterlevels}
\end{figure}

Due to the fact most station have no water level gauge at their exact location, we approximate their water level using the three closest gauges, weighted by their distance to the station. 
This data is recorded every minute and stored for one month. For our tests we use the data between \formatdate{20}{11}{2023} and \formatdate{20}{12}{2023}, although older data can be requested from the Bundesanstalt für Gew\"asserkunde and estimates for the future are done by the Bundesamt für Seeschifffahrt und Hydrographie.
We gain depth data for the stations from \cite{BSH.2023}. In cases where no measurement for the exact location is given, a closest point is used. 
In practical application a safety margin might be necessary, which we do not consider for now. In total, we arrive at around $8700$ different tidal states.

By dividing the data into the North Sea and and the Baltic Sea we receive two additional smaller data sets that could be used for further comparisons. As the DGzRS however does not need to plan for them independently but simultaneously for both we do not include this in our computational study.

\subsection{Setup for the study}
The code was written in Python (Version 3.11.2) with Gurobi (Version 10.0.0,  \cite{GurobiOptimization.2023}) and was executed on the High Performance Cluster of the RWTH Aachen\footnote{\url{https://web.archive.org/web/20240107173048/https://help.itc.rwth-aachen.de/service/rhr4fjjutttf/article/fbd107191cf14c4b8307f44f545cf68a/}}, using Intel Xeon Platinum 8160 Processors (2.1 GHz).

The three different approaches are named \textit{many-zones} (\cref{subsec:algo:heuristic:station}), \textit{better-tidal} (\cref{subsec:algo:heuristic:stationwater}) and \textit{best-tidal} (\cref{sec:exact_sol}).
For comparability the instances of all three approaches run on a single core CPU. Since runtime and RAM limitation is needed for the usage of the High Performance Cluster, we cap both. To ensure a solution by Gurobi, we set the \textit{TimeLimit} property of Gurobi to 6 hours and limit the total time for the jobs on the Cluster to 9 hours. 

All instances are initialized with the same $1000$ zones. Because this amount is too high for all approaches, we use the $k$-means clustering algorithm to shrink them down to a desired number $n_z \in \{1, \dots, 1000\}$. As \textit{better-tidal} is more difficult to solve than \textit{many-zones}, the approach \textit{many-zones} uses $n_z \in \{10,50,100\}$ zones for problem solving while \textit{better-tidal} uses $n_z \in \{10, 20, 30\}$ zones. Both approaches have access to 32GB RAM. Because \textit{best-tidal} is very hard to solve, we only use the values $n_z \in \{1, 2, 5\}$ for it with access to 64GM RAM. The fact that the job with $n_z=5$ did run out of memory led us to the decision to discard this job for to practical concerns.
To compare the different approaches we calculate all solutions on the objective function of \textit{best-tidal} on all $1000$ zones.
For evaluation we run each configuration with $10$ different seeds for the random parameters. This means in total we have $10 \cdot 3$ instances for every approach, which results in $90$ jobs overall.

\subsection{Prelininary computational results}
First, the build times of the models, as can be seen in \cref{fig:compstudy:results:2}, already limit the usability of \textit{best-tidal}.
\begin{figure}[h!tb]
    \centering
    \includegraphics[width=\textwidth]{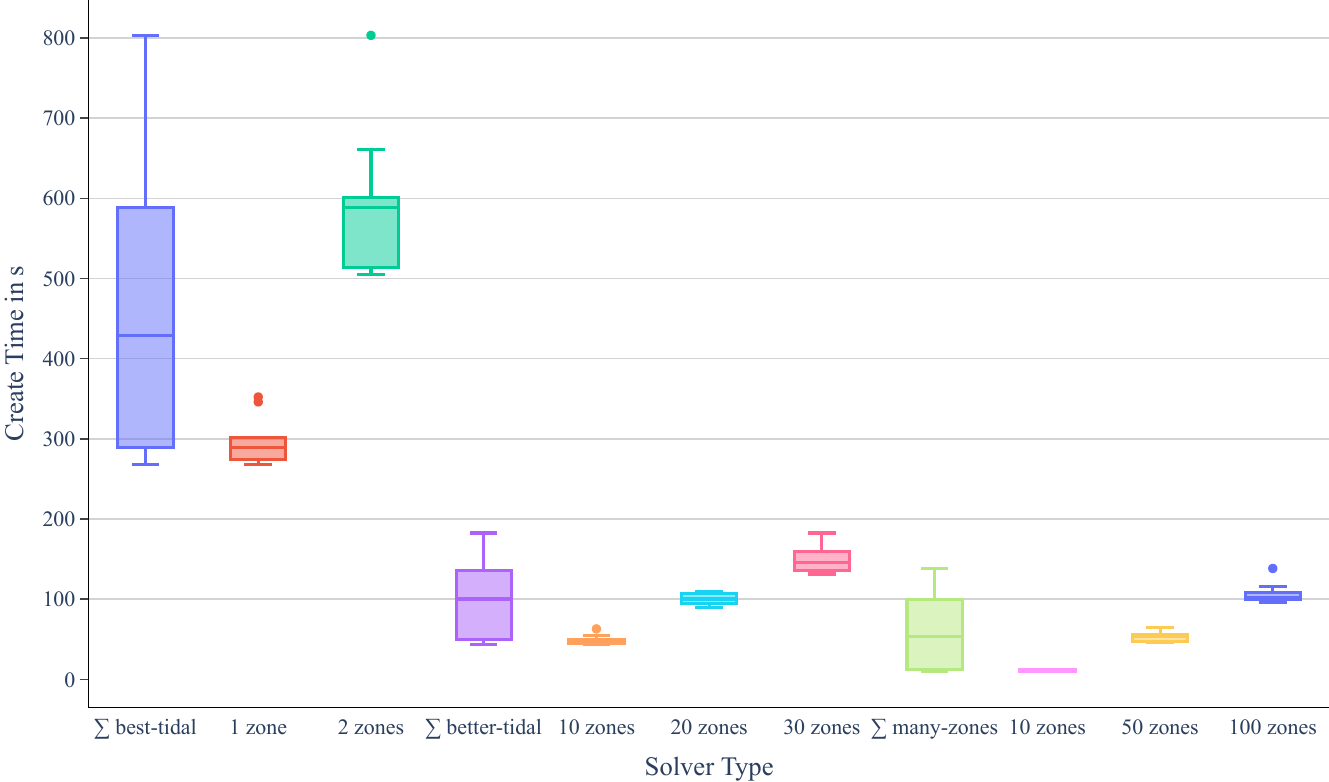}
    \caption{Time required to build the Gurobi Models and cluster the zones.}\label{fig:compstudy:results:2}
\end{figure}
This, combined with the huge memory requirement of \textit{best-tidal} as can be seen in \cref{fig:compstudy:results:3}, validates that running \textit{best-tidal} without zone clustering is not viable. 
\begin{figure}[h!tb]
    \centering
    \includegraphics[width=\textwidth]{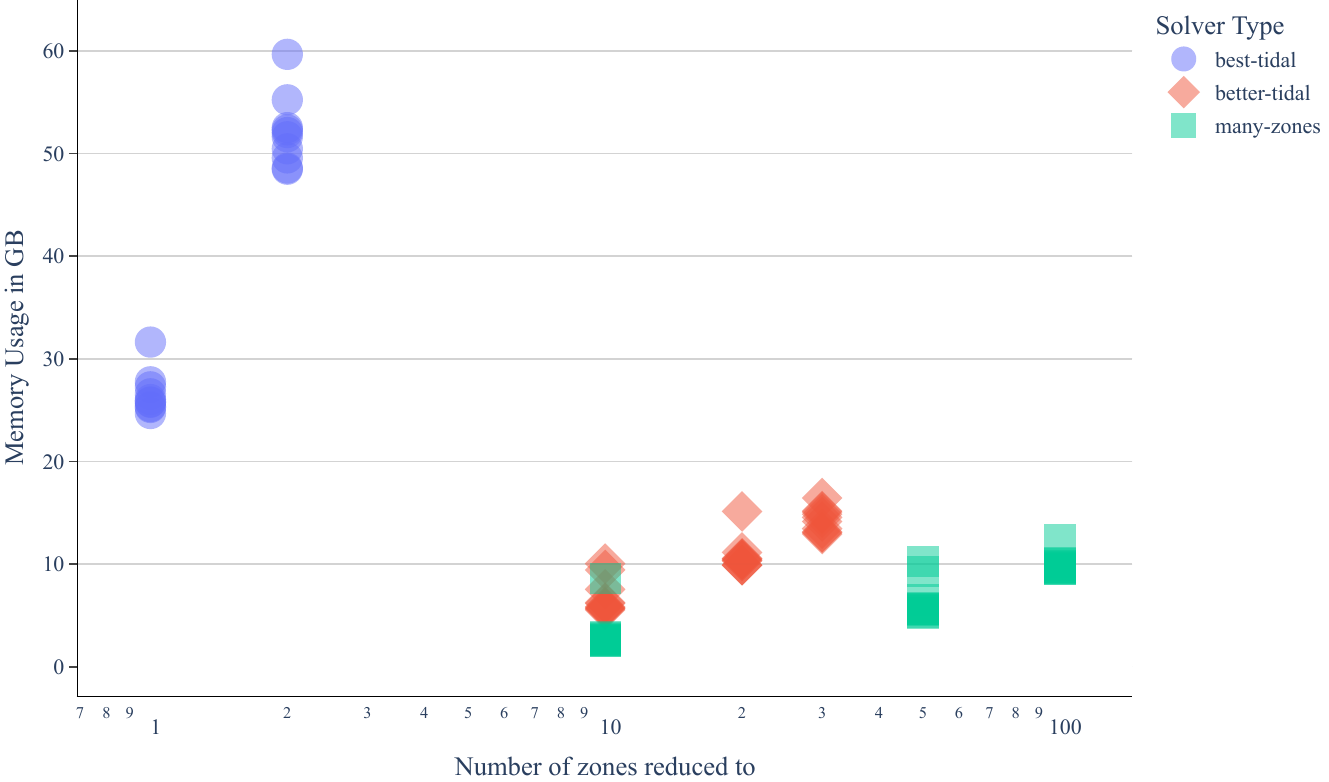}
    \caption{Memory usage of the complete instances with $1000$ zones and the complete tidal information. The $x$-axis is $\log$ scaled for better readability.} \label{fig:compstudy:results:3}
\end{figure}
Therefore it is necessary to compare the solution quality of different heuristics and different amounts of zone clusters, as for example $100$ zone clusters in \textit{many-zones} result in more memory efficiency and a lower run time than $2$ clusters in \textit{best-tidal}.
To do so, we extract the vessel-station assignment $M$ of every solution. We then validate the second part of the solution ($u$) by checking for all combinations of  each of the original $1000$ zones, each incident type and each tidal state (according to our best available tidal data) which of the stationed compatible vessels can respond the fastest. In this process we also confirm that all of the produced solutions for the different instances are feasible.

\section{Discussion}\label{sec:discussion}
The resulting objective values can be seen in \cref{fig:compstudy:results:1}. It is important to note that not every instance has the same optimal solution value.

\begin{figure}[h!tb]
    \centering
    \includegraphics[width=\textwidth]{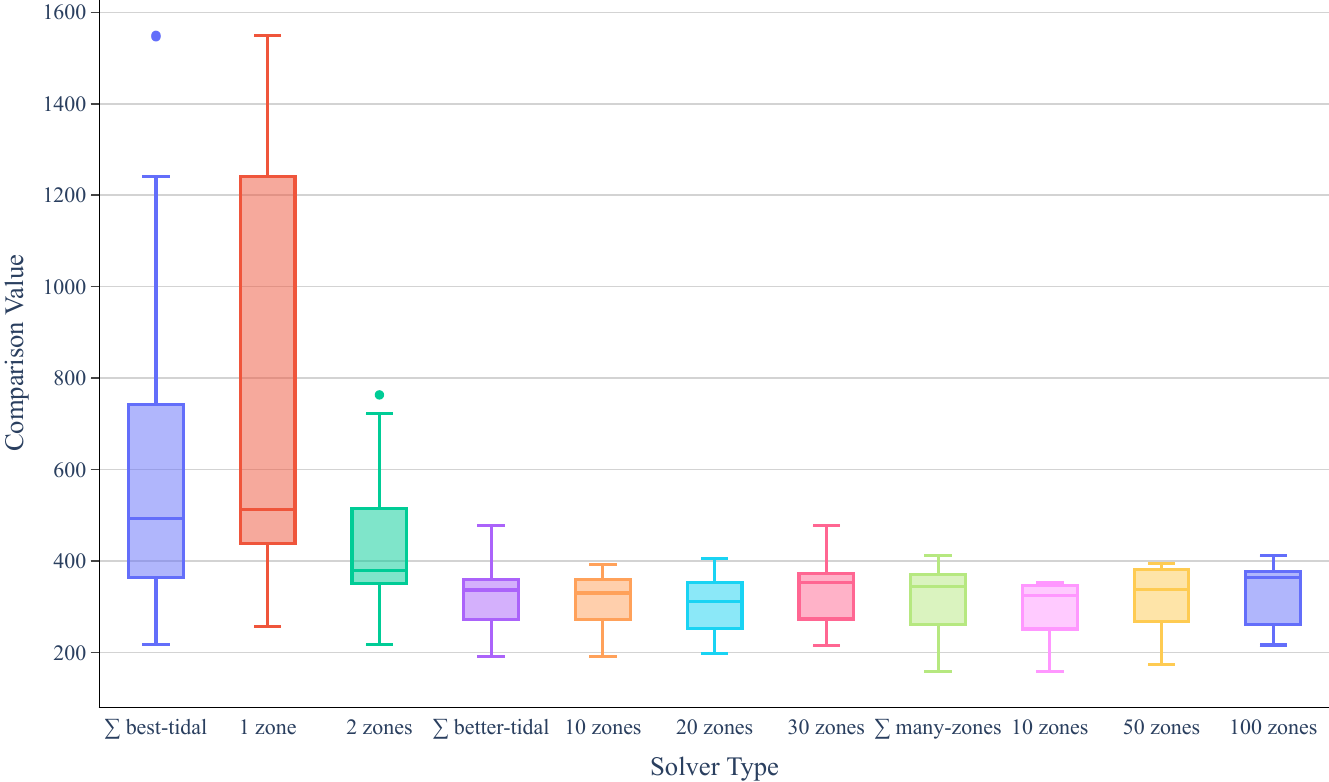}
    \caption{Objective Value in the complete instances with $1000$ zones and the complete tidal information.} \label{fig:compstudy:results:1}
\end{figure}

Looking at the values it is clear that in our instances we gain better results with a higher amount of zone clusters in exchange for a reduction in tidal states.
This is an be explained by the correlation between the tidal levels at the different stations, as noted before. Additionally we see that the difference between \textit{better-tidal} and \textit{many-zones} is rather small. This is most likely due to the fact that the vessels have quite similar draught, the vessel with the lowest value has $0.5$ meters while the one with the highest has $2.7$ meters. Furthermore, \textit{many-zones} is more memory efficient and managed, unlike \textit{better-tidal}, to sometimes find the optimal solution to the simplified problem within the time limit and therefore exited faster. We can also see that clustering to a small amount of zones does not seem to have a big influence on \textit{many-zones} and \textit{better-tidal}.

Note that currently the vessel-station-compatibility is randomized. This has to be replaced with data of the \ac{dgzrs} detailing if a vessel could be placed at a station, taking the vessel size and crew into account. For example a vessel can not be placed at a station if it is too large or if the locally available staff size is too low.
For non-randomized incident rates, it might be worthwhile to overhaul the concept of zones. It is more appropriate to divide the region into equally sized zones. Currently this can not be done because increasing the base amount of zones leads to convergence against the same incident rates with $k$-means clustering. Real world incident data would allow for equally sized zones without such statistical issues. However the exact amount of zones would be less controllable, since the input parameter for the zone creation would be zone size instead of zone amount.

At the moment tide data is used to determine availability of stations. Of course zones are also affected by the tides. It might be worthwhile to also give the zones a tide property and expand route calculation to note temporarily inaccessible zones. Moreover tides impacts the travel time. Considering all these aspects would lead to an even larger and more complex model.

To improve the accuracy of the results, better tide data could be used. A possibility is to enlarge the time horizon of tide data to one year and average it. Another possibility is to use forecast data to make the results more reliable for future planning.

\section{Conclusion}\label{sec:conclusion}

In this work, we built a mathematical model for the \acf{rcap}. While in general, \ac{rcap} is $\NP$-hard and formulating it as a \ac{mip} leads to large formulations, there are approaches to effectively reduce the problem size. We have shown that clustering of incidents as well as simplification of the possible tide states lead to significant run time reductions. The degree of simplification regarding the tide states can be varied.

In order to compare these approaches we set up a computational study in which we tested them under restricted resources in regards to memory and runtime. To do so we used real world data of the \acf{dgzrs} and other sources. Some of the data such as incidents and vessel-station compatibility was generated randomly due to lacking data. Furthermore, this lack of data is also the reason for not comparing our results to the solution currently used, as our results might be infeasible in practice and the real solution in turn infeasible for our instances.

Our computational study establishes the validity of our concept of simplified tide representation and clustering of the incident zones. As shown in \cref{fig:compstudy:results:1} there is no significant difference in quality between \textit{better-tidal} and \textit{many-zones}. Beyond that \textit{best-tidal} returns worse results. Since, in the context of our computational study, using more zones didi not lead to better results, using a $k$-means clustering algorithm (\cref{subsec:algo:heuristic:kmeans}) has shown to be a good approach.
Moreover we obtained feasible solutions, which indicates a practical ability of our model. It might therefore make sense to test apply the model to real world data of the \ac{dgzrs} to attempt improving their current assignment.

\subsection{Further research}
The model presented in this work can be expanded on in multiple ways. First of all, the estimation of incident risk can be further refined. Part of this would be to introduce an uncertainty factor so that calculated solutions are robust to changes in incident data.

%Second, our approach is an \ac{mip}, so there are a multitude of tools to improve the solution process both in terms of runtime and space. An example of this is using column generation to find the comparably small number of non-zero variables in an optimal solution. Another approach would be to not use the $y$ variables from the \ac{mip} in \cref{sec:exact_sol} at all. This could be achived by using a custom branch-and-bound approach on the $x$ variables which does not rely on a \ac{mip} but instead calculates a custom upper and lower bound. 
%To do so one could use a closest assignment for each of the zone, incident, water state combinations by iterating through all of the allowed stations with their corresponding vessels. For the lower bound one could assume that all allowed vessels are stationed at each of the unused stations, for the upper bound a simple greedy approach could be used. This approach would yield a pseudo polynomial runtime for a fixed number of vessels and stations. We did implement an algorithm of this type and it did produce promising results for random tidal data, but the real tidal data posed a problem for it.
%The advantage of the random data is that it has not too many similar values.

Second, our assumption is that every incident is isolated, meaning a vessel can respond to an infinite number of incidents. While the number of incidents suggests a daily average of less than six incidents across the entire region making it unlikely that a vessel is required at two location simultaneously, it might be worthwhile to balance the workload across the vessels available.

Third, in our model the assumption is that the amount of vessels is fixed. If this number is reduced, it might be interesting to analyze which vessel should be removed and what the increase in response times would be. Conversely, it is not trivial which improvement an additional vessel would make and where to place it.

Fourth, we only consider the haversine distance between harbors and zones. However, in the real-world factors, like fairway, depth of water, restricted areas and even the location of islands have to be considered when computing a route.  Implementing this would serve to make the modelling more valid to practitioners. Furthermore, as suggested by \cite{Siljander2015}, wind and wave patterns also influence rescue speed, which would lend itself to stochastic optimization.

Finally it would be interesting to see if there are yearly patterns in the incident data and how beneficial a seasonal reallocation of the rescue vessels would be.

\appendix

\section{Tide Correlation Data}
\label{app:sec:tidecorrelations}
\input{tide_correlations}

\begin{acronym}[ECU] 
    \acro{dgzrs}[DGzRS]{Deutsche Gesellschaft zur Rettung Schiffbrüchiger (German Maritime Search and Rescue Service)}
    \acro{x3c}[X3CP]{Exact Cover by 3-Sets Problem}
    \acro{rcap}[RCAP]{Rescue Craft Allocation Problem}
    \acro{frcap}[f-RCAP]{Feasibility Rescue Craft Allocation Problem}
    \acro{mip}[IP]{Integer Program}
    \acro{cavabb}[CAVABB]{Closest Assignment Vessel Allocation by Branch and Bound}
    \acro{nhn}[NHN]{Normalh\"ohennull}
    \acro{SAR}[SAR]{Search-and-Rescue}
\end{acronym}

\bibliographystyle{elsarticle-num-names} 
\bibliography{cas-refs}
\end{document}

%% file: table_citations.tex
\begin{table}[h!tb]
\setlength{\tabcolsep}{5pt}
    \centering
    \begin{tabular}{lccccc}
    \multirow{2}{*}{source {\color{gray}(comment)}} & given & vessel & \multirow{2}{*}{zones} & \multirow{2}{*}{IP} & \multirow{2}{*}{algorithm}\\
    & stations & types & & & \\\hline 
     
    Afshartous et al. (2009) \cite{Afshartous.2009} & $\sim$ & n & y & y & n  \\
    \multicolumn{6}{l}{\color{gray}Conversion of incident data into incident probabilities} \\
    
    Ai et al. (2015) \cite{Ai.2015} & y & $\sim$ & y & $\sim$ & y  \\
    \multicolumn{6}{l}{\color{gray}Comparison of heuristics for solution finding} \\
    
    Azofra et al. (2007) \cite{Azofra.2007} & y & n & y & $\sim$ & n \\
    \multicolumn{6}{l}{\color{gray}Basic model for assigning a single vessel} \\

    Chen et al. (2021) \cite{chen2021} & y & y & y & y & y \\
    \multicolumn{6}{l}{\color{gray}Separation into tactical and operational phase} \\

    Jin et al. (2021) \cite{jin2021} & n & y & y & n & y \\ 
    \multicolumn{6}{l}{\color{gray}Minimisation of construction and maintenance cost} \\
    
    Jung \& Yoo (2019) \cite{Jung2019} & $\sim$ & n & y & n & y \\
    \multicolumn{6}{l}{\color{gray}Consideration of islands and coastline in distance calculation} \\
    
    Feldens \& Chen (2020) \cite{FeldensFerrari.2020} & y & n & y & y & n \\
    \multicolumn{6}{l}{\color{gray}Maximisation of covered area in search and rescue operation} \\
    
    Hoernberger et al. (2020) \cite{Hornberger.2022} & y & $\sim$ & y & y & n \\
    \multicolumn{6}{l}{\color{gray}Inclusion of relocation costs in regards to current allocation} \\
    
    Karatas (2021) \cite{Karatas.2021} & y & y & y & y & n \\
    \multicolumn{6}{l}{\color{gray}Considers tradeoff between response time, working hours and budget} \\

    Ma et al. (2024) \cite{ma2024} & $\sim$ & n & $\sim$ & y & n \\
    \multicolumn{6}{l}{\color{gray}Robust optimization} \\

    Pelot et al. (2015) \cite{Pelot2015} & y & y & y & y & n \\
    \multicolumn{6}{l}{\color{gray}Multiple modelling approaches} \\
    
    Razi \& Karatas (2016) \cite{Razi.2016} & y & y & y & y & n \\
    \multicolumn{6}{l}{\color{gray}Inclusion of different incident types} \\
    
    Wagner \& Radovilsky (2012) \cite{Wagner.2012} & y & y & n & y & n \\
    \multicolumn{6}{l}{\color{gray}Development of model for practical use} \\
    
    Zhou et al. (2022) \cite{Zhou.2022} & y & n & y & n & n \\
    \multicolumn{6}{l}{\color{gray}Search and rescue from a game theoretical perspective}
    \end{tabular}
    
    \caption{Literature on rescue vessel placement. The entries indicates [y]es or [n]o, with [$\sim$] denoting special cases.}
    
    \label{tab:source:comp}
    
\end{table}

%% file: vartable.tex
\begin{sidewaystable}
\setlength{\tabcolsep}{5pt}
    \centering
    \begin{tabular}{l|l|l}
    Variable & Element of & Usage \\
     \hline 
    $\paramVesselAmount{i}$             & $ \N $            & amount of vessels of type $i$ available \\
    
    $\paramVesselIncident $             & $\mathcal{P}( \underline{\paramVessels} \times \underline{\paramTypes} )$          & $(i,k) \in \paramVesselIncident$ if vessel type $i$ has the equipment required for incident type $k$ \\
    
    $\paramVesselStation $              & $\mathcal{P}( \underline{\paramVessels} \times \underline{\paramStations} )$           & $(i,j) \in \paramVesselStation$ if vessel type $i$ can be placed at station $j$ \\
    
    $\paramDistanceHarbourZone{j}{r} $  & $ \R^+ $                  & distance between station $j$ and zone $r$ \\
    
    $e$                                 & $\mathcal{P}(\underline{\paramVessels} \times \underline{\paramStations})$         & state of tides in which exactly the vessel-station combinations $(i,j) \in e$ are available  \\
    
    $\paramTimes$                       & $ \mathcal{P}(\mathcal{P}(\underline{\paramVessels} \times \underline{\paramStations}))$    & uncertainty set of tide states with non-zero probability \\
    $ \paramTypes$                      & $ \N $                    & amount of incident types \\
    $ i$                                & $\underline{\paramVessels} $ & index referring to a certain vessel type \\
    $ j$                                & $\underline{\paramStations} $ & index referring to a certain station \\
    $ k$                                & $\underline{\paramTypes} $ & index referring to a certain incident type \\
    $ \paramStations$                   & $\N$                      & amount of stations \\
    $M$                                 &                           & $b$-matching representing the vessel-station-allocation \\
    $ \paramVessels$                    & $\N$                      & amount of vessel types \\
    $\paramCorrectedWaterlevel{e}$        & $[0,1]$                   & probability of tidal state $e$ occurring \\
    $\paramIncidentZoneProbability{k}{r}$ & $[0,1]$                 & frequency of incident type $k$ occuring in zone $r$ \\
    $ r$                                & $\underline{\paramZones}$ & index referring to a certain zone \\
    $ \paramVesselHarbourZone$          & $\mathcal{P} ( \underline{\paramVessels} \times \underline{\paramStations} \times \underline{\paramZones} )$                          & $(i,j,r) \in \paramVesselHarbourZone$ if vessel type $i$ can travel from station $j$ to zone $r$ \\
    $u$                                 &                           & the vessel of station $u(k,r,e)$ responds to incident-zone-tide combination $(k,r,e)$ \\
    $\paramVesselSpeed{i}$              & $\R^+ $                   & speed of vessel type $i$ \\
    $\paramIncidentSevereity{k}$        & $\R^+ $                   & severity of incident type $k$ \\
    $\mathcal{X}$                       & $\mathcal{P}( \underline{\paramVessels} \times \underline{\paramStations} )$  & set of all possible vessel-station combinations \\
    $\mathcal{Y}$                       & $\mathcal{P}( \underline{\paramVessels} \times \underline{\paramStations} \times \underline{\paramTypes} \times \underline{\paramZones}\times \underline{\paramTimes} )$  & set of all possible vessel-station-incident-zone-tide combinations \\
    $\paramZones$                       & $\N$                      & amount of zones \\
    $ \delta_i$                         & $\R^+ $                   & draught of vessel type $i$
    \end{tabular}
    \caption{Overview variables where $\underline{x} \coloneqq \{1,\;\dots,\;x\}$ for $x \in \mathbb{N}$}
    \label{tab:variables}
\end{sidewaystable}

%% file: tide_correlations.tex
\resizebox{\textwidth}{!}{
\begin{tabular}{lrrrrrrrrrrrrrrrrrrrrrrrrrrrrrrrr}
\toprule
 & 9610010 & 9610015 & 9610020 & 9610025 & 9610035 & 9610040 & 9610045 & 9610050 & 9610066 & 9610070 & 9610075 & 9610080 & 9630007 & 9630008 & 9640015 & 9650024 & 9650030 & 9650040 & 9650043 & 9650070 & 9650073 & 9650080 & 9670046 & 9670050 & 9670055 & 9670063 & 9670065 & 9670067 & 9690077 & 9690078 & 9690085 & 9690093 \\
\bottomrule
\end{tabular}
}

\resizebox{\textwidth}{!}{
\begin{tabular}{lrrrrrrrrrrrrrrrrrrrr}
\toprule
 & 9340020 & 9340030 & 9360010 & 9390010 & 9410010 & 9510010 & 9510060 & 9510063 & 9510066 & 9510070 & 9510075 & 9510095 & 9510132 & 9530010 & 9530020 & 9550021 & 9570010 & 9570040 & 9570050 & 9570070 \\
\midrule
9340020 & {\cellcolor[HTML]{D9D9D9}} \color[HTML]{000000} 1.00 & {\cellcolor[HTML]{D9D9D9}} \color[HTML]{000000} 0.99 & {\cellcolor[HTML]{DADADA}} \color[HTML]{000000} 0.99 & {\cellcolor[HTML]{DBDBDB}} \color[HTML]{000000} 0.94 & {\cellcolor[HTML]{DBDBDB}} \color[HTML]{000000} 0.92 & {\cellcolor[HTML]{E2E2E2}} \color[HTML]{000000} 0.63 & {\cellcolor[HTML]{EBEBEB}} \color[HTML]{000000} 0.34 & {\cellcolor[HTML]{DDDDDD}} \color[HTML]{000000} 0.83 & {\cellcolor[HTML]{E0E0E0}} \color[HTML]{000000} 0.74 & {\cellcolor[HTML]{DDDDDD}} \color[HTML]{000000} 0.84 & {\cellcolor[HTML]{DDDDDD}} \color[HTML]{000000} 0.84 & {\cellcolor[HTML]{E4E4E4}} \color[HTML]{000000} 0.59 & {\cellcolor[HTML]{E9E9E9}} \color[HTML]{000000} 0.43 & {\cellcolor[HTML]{E8E8E8}} \color[HTML]{000000} 0.43 & {\cellcolor[HTML]{EFEFEF}} \color[HTML]{000000} 0.23 & {\cellcolor[HTML]{EBEBEB}} \color[HTML]{000000} 0.36 & {\cellcolor[HTML]{EBEBEB}} \color[HTML]{000000} 0.36 & {\cellcolor[HTML]{F8F8F8}} \color[HTML]{000000} 0.01 & {\cellcolor[HTML]{EFEFEF}} \color[HTML]{000000} 0.24 & {\cellcolor[HTML]{FAFAFA}} \color[HTML]{000000} -0.03 \\
\bottomrule
\end{tabular}}